\documentclass[11pt,a4paper]{article}
\usepackage{latexsym,amsfonts,amsmath,amsthm,graphics}
\usepackage[normalem]{ulem}
\usepackage{enumitem}
\usepackage[usenames,dvipsnames,svgnames,table]{xcolor}

\usepackage{pdfsync}

\newtheorem{theorem}{Theorem}
\newtheorem{lemma}[theorem]{Lemma}
\newtheorem{corollary}[theorem]{Corollary}
\newtheorem{proposition}[theorem]{Proposition}

\newtheorem{remark}[theorem]{Remark}

\newenvironment{proofof}[1]{\begin{trivlist}
    \item[\hskip\labelsep{\it Proof of {#1}.}]}{$\hfill\Box$\end{trivlist}}

\newcommand{\satop}[2]{\stackrel{\scriptstyle{#1}}{\scriptstyle{#2}}}

\newcommand{\bsgamma}{{\boldsymbol{\gamma}}}

\newcommand{\bsl}{\boldsymbol{l}}

\newcommand{\bsb}{\boldsymbol{b}}
\newcommand{\bsx}{\boldsymbol{x}}
\newcommand{\bsh}{\boldsymbol{h}}

\newcommand{\bsz}{\boldsymbol{z}}

\newcommand{\bsU}{\boldsymbol{U}}
\newcommand{\bsone}{\boldsymbol{1}}

\newcommand\setu{{\mathfrak{u}}}

\newcommand{\icomp}{\mathtt{i}}
\newcommand{\bszero}{\boldsymbol{0}}
\newcommand{\rd}{\,\mathrm{d}}
\newcommand{\NN}{\mathbb{N}}
\newcommand{\ZZ}{\mathbb{Z}}

\newcommand{\RR}{\mathbb{R}}

\newcommand{\calH}{\mathcal{H}}

\newcommand{\EE}{\mathbb{E}}

\newcommand{\abs}[1]{\left\vert#1\right\vert}

\newcommand{\bbone}{\mathbb{I}}
\newcommand{\calP}{\mathcal{P}}
\newcommand{\calA}{\mathcal{A}}
\newcommand{\calO}{{\mathcal{O}}}
\newcommand{\calZ}{{\mathcal{Z}}}

\newcommand{\bbZ}{{\mathbb{Z}}}
\newcommand{\bbN}{{\mathbb{N}}}

\allowdisplaybreaks

\begin{document}

\title{Lattice rules with random $n$ achieve nearly the optimal $\mathcal{O}(n^{-\alpha-1/2})$ error
independently of the dimension}

\author{Peter Kritzer, Frances Y. Kuo, Dirk Nuyens, Mario Ullrich}

\date{\today}

\maketitle

\begin{abstract}
We analyze a new random algorithm for numerical integration of $d$-variate
functions over $[0,1]^d$ from a weighted Sobolev space with
dominating mixed smoothness $\alpha\ge 0$ and product weights
$1\ge\gamma_1\ge\gamma_2\ge\cdots>0$, where the functions are
continuous and periodic when $\alpha>1/2$. The algorithm is based
on rank-$1$ lattice rules with a random number of points~$n$. For the case
$\alpha>1/2$, we prove that the algorithm achieves almost the optimal
order of convergence of $\mathcal{O}(n^{-\alpha-1/2})$, where the implied
constant is independent of the dimension~$d$ if the weights satisfy
$\sum_{j=1}^\infty \gamma_j^{1/\alpha}<\infty$. The same rate of
convergence holds for the more general case $\alpha>0$ by adding a random
shift to the lattice rule with random $n$. This shows, in particular, that
the exponent of strong tractability in the randomized setting equals
$1/(\alpha+1/2)$, if the weights decay fast enough. We obtain a lower
bound to indicate that our results are essentially optimal.

This paper is a significant advancement over previous related works with
respect to the potential for implementation and the independence of error
bounds on the problem dimension. Other known algorithms which achieve the
optimal error bounds, such as those based on Frolov's method, are very
difficult to implement especially in high dimensions. Here we adapt a
lesser-known randomization technique introduced by Bakhvalov in 1961. This
algorithm is based on rank-$1$ lattice rules which are very easy to
implement given the integer generating vectors. A simple probabilistic
approach can be used to obtain suitable generating vectors.
\end{abstract}


\section{Introduction} \label{sec:intro}

We study the problem of numerical integration of $d$-variate functions,
i.e., the approximation of
\[
 I_d(f) \,:=\,
 \int_{[0,1]^d} f(\bsx)\rd\bsx
\]
for $f$ in a weighted Sobolev space $\calH_{d,\alpha,\bsgamma}$ with
smoothness parameter $\alpha\ge 0$ and product weights $\bsgamma =
(\gamma_j)_{j\ge1}$ (details are provided below), where the functions are
continuous and periodic when $\alpha>1/2$ (also known in this case
as a weighted ``Korobov space'' in recent literature). To this end we use
a randomized algorithm $M_n$ that employs at most $n\in\NN$ function
evaluations. The basic building blocks of the algorithm $M_n$ are
so-called rank-$1$ lattice rules of the form
\begin{equation}\label{eq:lr}
Q_{d,p,\bsz}(f) \,:=\,
 \frac{1}{p}\sum_{k=0}^{p-1} f\left(\left\{\frac{k\bsz}{p}\right\}\right),
\end{equation}
where $p$ is a prime number, $\bsz\in\bbZ^d$ is known as the generating
vector, and $\{x\}$ denotes the fractional part of a real number $x$ and
is applied componentwise to a vector. Lattice rules as in \eqref{eq:lr}
are very well studied in the field of quasi-Monte Carlo methods (see
\cite{DKS13} and \cite{SJ94} for overviews on lattice rules, and
\cite{DP10,N92} for further introductions to the field of quasi-Monte
Carlo methods). Here, however, we do \textit{not} study one fixed
quasi-Monte Carlo rule, but an algorithm based on randomly choosing one of
a certain set of lattice rules. Indeed, the random algorithm $M_n$, for
given $n\in\NN$, is defined by $Q_{d,p,\bsz}$ with randomly chosen prime
$p\in\{\lceil n/2\rceil+1,\dots,n\}$ and $\bsz\in\{1,\dots,p-1\}^d$. To be
precise, we choose a random prime number $p$ in the given range and then a
random generating vector $\bsz$ from a certain set $\calZ_p$ of ``good''
generating vectors (see \eqref{eq:Zp-def} below). We call such an
algorithm $M_n$ a \emph{randomized lattice algorithm}.

An algorithm of such a form was first analyzed by Bakhvalov in~\cite{B61}
(see also~\cite{B59} for general results on randomized algorithms)
for generating vectors of a special form commonly known as the
``Korobov type'', i.e., $\bsz=(1,z,\dots,z^{d-1})$ for some $z\in\NN$.
Bakhvalov proved that this algorithm has almost the optimal order of
convergence in a Sobolev space with dominating mixed smoothness, but
the error bound depends on the dimension $d$.

A significant advancement of this paper is that we prove the existence of
good generating vectors, not restricted to the Korobov type, such that
our randomized lattice algorithm $M_n$ achieves almost the optimal order
of convergence in weighted Sobolev spaces, with the error bound
independent of $d$ under a summability condition on the weights $\bsgamma$
that is common in this field of research.

The advantage over previous related works lies in the potential for
implementation. Other known algorithms which achieve the optimal error
bounds, such as those based on Frolov's method (see, e.g.,
\cite{Fr76,KOU17,KN16,MU15,MU17,UU16}), are very difficult to implement
especially in high dimensions. For our algorithm, a simple probabilistic
approach can be used to obtain suitable generating vectors, see
Remark~\ref{rem:imp} below.

More precisely, we analyze the \emph{randomized \textnormal{(}worst
case\textnormal{)} error} for our randomized lattice algorithm $M_n$ in
the unit ball of $\calH_{d,\alpha,\bsgamma}$, defined by
\begin{equation} \label{eq:rand-err}
 e^{\rm ran}_{d,\alpha,\bsgamma}(M_n) \,:=\,
 \sup_{\substack{f\in \calH_{d,\alpha,\bsgamma}\\ \|f\|_{d,\alpha,\bsgamma}\le 1}}
 \EE\Big[\big| M_n(f)- I_d (f) \big|\Big].
\end{equation}
The norm $\|\cdot\|_{d,\alpha,\bsgamma}$ will be defined in
Section~\ref{sec:spaces}. The details of the expectation will be made
clear when we formally specify our randomized lattice algorithm in
Section~\ref{sec:rand}. We have three main theorems in this paper. In
Theorem~\ref{thm1} below, we prove for $\alpha>1/2$ (thus all functions
are continuous and periodic) and $n$ sufficiently large that
\[
  e^{\rm ran}_{d,\alpha,\bsgamma}(M_n)
  \,=\, \calO( n^{-a-1/2})
  \qquad\mbox{for $a<\alpha$ arbitrarily close to $\alpha$},
\]
where the implied constant is independent of $d$ if
\begin{equation} \label{eq:sum-weights}
  \sum_{j=1}^\infty \gamma_j^{1/\alpha} \,<\, \infty.
\end{equation}
(In the literature, the exponent in~\eqref{eq:sum-weights} sometimes differs
by a factor of two, depending on how $\alpha$ and $\bsgamma$ enter the
definition of the norm, see Section~\ref{sec:spaces}.)

Additionally we will analyze the \emph{randomized lattice algorithm
with shift}, which is defined by
\[
 \widetilde{M}_n(f) \,:=\, M_n\bigl( f(\{\cdot+\bsU\}) \bigr),
\]
where $\bsU$ is a random variable that is uniformly distributed on
$[0,1]^d$, and the braces mean to take the fractional part as in
\eqref{eq:lr}. Again, the algorithm uses at most $n$ function evaluations.
The advantage of the algorithm $\widetilde{M}_n$ is that it can treat the
case $\alpha \in(0,1/2]$, i.e., when the integrands are not necessarily
continuous. In Theorem~\ref{thm2} below, we prove for $\alpha>0$ that
\[
  e^{\rm ran}_{d,\alpha,\bsgamma}(\widetilde{M}_n)
  \,=\, \calO( n^{-a-1/2})
  \qquad\mbox{for $a<\alpha$ arbitrarily close to $\alpha$},
\]
where again the implied constant is independent of $d$ if
\eqref{eq:sum-weights} holds. Moreover, the result also holds with the
randomized error replaced by the root-mean-squared error of
$\widetilde{M}_n$, which is larger.

Then, we show that the upper bounds for the algorithms $M_n$ and
$\widetilde{M}_n$ are essentially best possible, by proving in
Theorem~\ref{thm3} the lower bound
\[
  \frac{\gamma_1}{2} \frac{\sqrt{\log n}}{n^{\alpha+1/2}}
\]
for the randomized errors of the above algorithms.

The presented upper bounds are almost optimal because the optimal order of
convergence that can be achieved by any randomized algorithm which uses
only function values is $\Theta( n^{-\alpha-1/2})$, even in the unweighted
situation. This has been proven only recently by one of the
authors~\cite{MU17} using the algorithm of~\cite{KN16}. However, the
algorithm in~\cite{MU17} is quite impractical in high dimensions, in
contrast to the algorithm that will be analyzed in the following. This
mainly comes from the fact that the involved point sets, i.e., certain
subsets of irrational lattices in $\RR^d$, seem to be very hard to
implement, see, e.g.,~\cite{KOU17,UU16}. Moreover, the presented upper
bound in~\cite{MU17} is at least exponential in $d$ and probably not
improvable, see~\cite[Remark 3.2]{MU17}.

The optimal order $\Theta( n^{-\alpha-1/2})$ in the randomized setting,
which holds for arbitrary $\alpha\ge0$, should be compared to the optimal
order $\Theta( n^{-\alpha}(\log n)^{(d-1)/2})$ in the deterministic
setting, which can only hold if $\alpha>1/2$, see \cite{By85, Fr76} or
\cite{Te93,UU16} and the references therein. For $\alpha\le1/2$,
deterministic methods do not converge at all in the worst case setting, as
the functions are possibly not continuous. A tutorial on the proof of the
upper bound can be found in~\cite{MU15}. By now there are many known
constructions for \emph{optimal} algorithms if $\alpha\le1$, see,
e.g.,~\cite{DP10,N92,NW10}, while for $\alpha>1$ they are still rare, see
\cite{Fr76,GSY16}. By introducing weights in these function spaces, one
can achieve the rate $\calO(n^{-a})$ for $a<\alpha$, with the implied
constant independent of $d$, see~\cite[Theorem~3]{SW01}. For a
component-by-component construction that achieves this bound, see, e.g.,
\cite{DKS13,K03}.

The rest of the paper is structured as follows. In Section \ref{sec:det},
we define the function space, summarize some known results from the
deterministic error setting and show a few auxiliary results that will be
needed in this paper. In Section \ref{sec:rand} we state and prove the
main results, Theorems \ref{thm1} and \ref{thm2}, while in Section
\ref{sec:lower} we show a lower bound for the specific algorithms
considered in this paper, see Theorem \ref{thm3}. The independence of the
error bounds on the dimension is the essence of \emph{strong
tractability}, and we will provide a brief discussion about this in
Section~\ref{sec:tractresults}. In particular, again under a suitable
summability condition on the weights, the exponent of strong tractability
in the randomized setting equals the optimal value of $1/(\alpha+1/2)$. We
end the paper with a short conclusion in Section \ref{sec:concl}.

\goodbreak

\section{Previous and auxiliary results} \label{sec:det}

In this section we define the function space and review some known
results for the error of lattice rules in the deterministic setting (this
includes the definitions necessary for the analysis), and we also prove a
few helpful results.

\subsection{Function spaces}\label{sec:spaces}

Let us now define the function spaces under consideration in this paper.
These spaces are so-called \emph{Sobolev spaces with dominating mixed smoothness}, where we assume
(product) weights as introduced by Sloan and Wo\'{z}niakowski in
\cite{SW98}.
The basic idea of weighted spaces is that the weights,
occurring in the norm of the space, allow us to model different influence
of different coordinates on the integration problem, where larger weights
mean more influence and smaller weights mean less influence.

Let $\alpha\ge0$, let $\bsgamma=(\gamma_j)_{j\ge 1}$ be a non-increasing
sequence of positive weights bounded by 1, and let $d\in\NN$. We denote by
$\calH_{d,\alpha,\bsgamma}$ the Sobolev space of functions
defined on
$[0,1]^d$, with finite norm
\[
  \|f\|_{d,\alpha,\bsgamma}
  \,:=\, \left(
  \sum_{\bsh\in\bbZ^d} \big|r_{\alpha,\bsgamma}(\bsh)\,\hat{f}(\bsh)\big|^2\right)^{1/2}\,,
\]
where $\hat{f}(\bsh):=\int_{[0,1]^d}f(\bsx)\, e^{-2\pi i
\bsh\cdot\bsx}\rd\bsx$, with ``$\cdot$'' denoting the usual Euclidean
inner product, and
\[
 r_{\alpha,\bsgamma}(\bsh) \,=\, \prod_{j=1}^d r_{\alpha,\gamma_j} (h_j),
\]
with
\[
 r_{\alpha,\gamma_j}(h_j) \,:=\, \max\left\{1,\,\frac{\abs{h_j}^{\alpha}}{\gamma_j}\right\}.
\]

It is known that $\calH_{d,\alpha,\bsgamma}\subseteq L_2 ([0,1]^d)$ (with
equality for $\alpha=0$), and that its elements can be expressed in terms
of their Fourier series, i.e.,
\[
 f(\bsx)
 \,=\, \sum_{\bsh\in\ZZ^d} \hat{f}(\bsh)\, e^{2\pi\icomp \bsh\cdot\bsx}
 \quad\mbox{for}\quad f\in \calH_{d,\alpha,\bsgamma}.
\]
Note that the convergence of the Fourier series holds pointwise if
$\alpha>1/2$, and in this case $\calH_{d,\alpha,\bsgamma}$ consists of
continuous and periodic functions.
In the recent literature on lattice rules,
$\calH_{d,\alpha,\bsgamma}$ with $\alpha>1/2$ is
often called \emph{weighted Korobov space}. If
$\alpha\le1/2$, then the Fourier series of a function in
 $\calH_{d,\alpha,\bsgamma}$ does not necessarily converge pointwise,
and the above equation is to be understood almost everywhere, which is
enough for our purposes.

It is straightforward to check that, for $\alpha\in\NN$,
\[
\|f\|_{d,\alpha,\bsgamma}^2 \,=\,
\sum_{\setu\subseteq \{1:d\}} (2\pi)^{-2\alpha|\setu|} \Bigg(\prod_{j\in\setu}\,\gamma_j^{-2}\Bigg)
	\Bigg( \int_{[0,1]^{|\setu|}} \Bigg| \int_{[0,1]^{d-|\setu|}}
	\Bigg(\prod_{j\in\setu}\frac{\partial}{\partial x_j}\Bigg)^\alpha f(\bsx)\,
	\rd\bsx_{\{1:d\}\setminus\setu} \Bigg|^2 \rd\bsx_{\setu} \Bigg),
\]
where $\{1:d\}:=\{1,\dots,d\}$ and $\bsx=(x_1,\dots,x_d)$ with
$\bsx_{\setu}$ and $\bsx_{\{1:d\}\setminus\setu}$ denoting the projection
of $\bsx$ onto the coordinates in $\setu$ and $\{1:d\}\setminus\setu$,
respectively.

\begin{remark} If we replace the definition of
$r_{\alpha,\gamma_j}(h_j)$ by
\[
 r^{*}_{\alpha,\gamma_j}(h_j) \,:=\, \sqrt{1+\frac{|2\pi h_j|^{2\alpha}}{\gamma_j^2}},
\]
then, for $\alpha\in\NN$, we obtain the norm
\[
 \|f\|_{*}^2
 \,=\, \sum_{\setu\subseteq\{1:d\}}
 \Bigg(\prod_{j\in\setu}\gamma_j^{-2}\Bigg) \;
 \Bigg\|\Bigg(\prod_{j\in\setu}\frac{\partial}{\partial x_j}\Bigg)^\alpha f\,\Bigg\|_{L_2([0,1]^d)}^2
 \,=\, \sum_{\bsb\in\{0,1\}^d} \bsgamma^{-2\bsb}\, \| D^{\alpha\bsb}f\|_{L_2([0,1]^d)}^2,
\]
with $\bsgamma^{-2\bsb}=\prod_{j=1}^d \gamma_j^{-2b_j}$, which recovers
the standard norm for the Sobolev space when $\bsgamma\equiv\bsone$. We
could obtain the latter upper error bounds for this norm in the same way,
since $r_{\alpha,\gamma_j}(h_j)\le r^*_{\alpha,\gamma_j}(h_j)$, which
implies $\|f\|_{d,\alpha,\bsgamma}\le\|f\|_*$.
\end{remark}

\subsection{Lattice rule error}

For a prime $p$ and $\bsz=(z_1,\ldots,z_d)\in\{1:p-1\}^d$,
where $\{1:p-1\} := \{1,2,\ldots,p-1\}$, it easily follows that the
error of a (rank-1) lattice rule $Q_{d,p,\bsz}$ is given by
\begin{equation}\label{eq:lattice-error}
Q_{d,p,\bsz}(f) - I_d(f) \,=\,
  \sum_{\satop{\bsh\in\bbZ^d\setminus\{\bszero\}}{\bsh\cdot\bsz \equiv_p 0}}
  \hat{f}(\bsh)\, ,
\end{equation}
provided that $\sum_{\bsh \in \ZZ^d} |\hat{f}(\bsh)| < \infty$, where
$\bsh\cdot\bsz \equiv_p 0$ means that $\bsh\cdot\bsz$ is congruent to zero
modulo~$p$. Hence, the performance of a lattice rule $Q_{d,p,\bsz}$
depends solely on the structure of the \emph{dual lattice}, i.e., the set
of all $\bsh\in\ZZ^d$ such that $\bsh\cdot\bsz \equiv_p 0$. There are
several measures, or \emph{figures of merit}, for the quality of a lattice
rule, see, e.g.~\cite{Ni78,N92}.

One figure of merit is
\[
  P_{\alpha,\bsgamma}(p,\bsz) \,:=\,
  \sum_{\satop{\bsh\in\bbZ^d\setminus\{\bszero\}}{\bsh\cdot\bsz \equiv_p 0}}
  \frac{1}{r_{\alpha,\bsgamma}(\bsh)},
\]
for which the unweighted version was defined independently by Hlawka in
\cite{H62} and Korobov in \cite{K59}. This quantity is the worst case
error in the class of functions $E_{\alpha,\bsgamma}$, which consists of
all functions $f$ with $|\hat{f}(\bsh)|\le r_{\alpha,\bsgamma}(\bsh)^{-1}$
for all $\bsh\in\ZZ^d$. For the weighted Sobolev space that is considered
in this paper, we clearly obtain by the Cauchy--Schwarz inequality that
\begin{equation}\label{eq:error-det}
\begin{split}
\abs{Q_{d,p,\bsz}(f) - I_d(f)}
	\,&\le\, \|f\|_{d,\alpha,\bsgamma}\,\Bigg(
		\sum_{\satop{\bsh\in\bbZ^d\setminus\{\bszero\}}{\bsh\cdot\bsz \equiv_p 0}}
		\frac{1}{[r_{\alpha,\bsgamma}(\bsh)]^2}
	\Bigg)^{1/2}\\
	\,&=\, \|f\|_{d,\alpha,\bsgamma}\,\sqrt{P_{2\alpha,\bsgamma^2}(p,\bsz)},
\end{split}
\end{equation}
where we used that
$[r_{\alpha,\bsgamma}(\bsh)]^2=r_{2\alpha,\bsgamma^2}(\bsh)$. It is easily
seen that equality holds in \eqref{eq:error-det} for some (worst case)
function $f$, and hence we conclude that the (deterministic) worst case
error of a single lattice rule $Q_{d,p,\bsz}$ is precisely
\begin{equation} \label{eq:wor-err}
 e^{\rm wor}_{d,\alpha,\bsgamma}(Q_{d,p,\bsz}) \,:=\,
 \sup_{\substack{f\in \calH_{d,\alpha,\bsgamma}\\ \|f\|_{d,\alpha,\bsgamma}\le 1}}
 \left| Q_{d,p,\bsz}(f)- I_d (f) \right|
 \,=\,
 \sqrt{P_{2\alpha,\bsgamma^2}(p,\bsz)}.
\end{equation}

Another relevant figure of merit is
\[
\rho_{\alpha,\bsgamma}(p,\bsz) \,:=\,
\min_{\satop{\bsh\in\bbZ^d\setminus\{\bszero\}}{\bsh\cdot\bsz \equiv_p 0}}\,
r_{\alpha,\bsgamma}(\bsh),
\]
which is a weighted version of the \emph{Zaremba index} for higher
smoothness. Clearly,
\begin{equation}\label{eq:Zaremba}
 \frac{1}{\rho_{\alpha,\bsgamma}(p,\bsz)} \,<\, P_{\alpha,\bsgamma}(p,\bsz).
\end{equation}

\subsection{The existence of good generating vectors}

Here, we show by a standard averaging argument that there exist generating
vectors which make the worst case error of the corresponding lattice rule
small. In addition, we show that many such vectors exist, which will be
essential for the proof of our main result. Recall that, for $\bsh\in
\ZZ^d$, we have
\[
 r_{\alpha,\bsgamma}(\bsh) \,=\, \prod_{j=1}^d
\max\Bigl\{1,\, |h_j|^\alpha/\gamma_j\Bigr\}.
\]

We start with some auxiliary results.

\begin{lemma}\label{lem:sum-r}
Let $d\in\NN$, $\beta>1$, and $\bsgamma\in(0,1]^\NN$.
Then we have
\begin{equation}\label{eq:lemsum-r}
\sum_{\bsh\in\ZZ^d} \frac{1}{r_{\beta,\bsgamma}(\bsh)}
\,=\, \prod_{j=1}^d\left(1+2\gamma_j\, \zeta(\beta)\right),
\end{equation}
where $\zeta$ denotes the Riemann zeta function.
\end{lemma}
\begin{proof}
The result follows easily by the definition of the function $r_{\beta,\bsgamma}$.
\end{proof}

The quantity on the right-hand side of \eqref{eq:lemsum-r} will be crucial
in the following computations. To simplify notation, we define, for
$d\in\NN$, $\beta>1$, and $\bsgamma\in(0,1]^\NN$,
\begin{equation}\label{eq:V}
V_d(\beta,\bsgamma) \,:=\,
 3 \sum_{\bsh\in\ZZ^d} \frac{1}{r_{\beta,\bsgamma}(\bsh)} \,=\,
 3 \prod_{j=1}^d\left(1+2\gamma_j\, \zeta(\beta)\right)
 \,\le\, 3 \exp\left(2\,\zeta(\beta)\sum_{j=1}^d\gamma_j \right),
\end{equation}
where we used $1+x\le e^x$ for $x\in\RR$. Note that $ V_d(\beta,\bsgamma)$
is bounded independently of $d$ for all $\beta>1$ if $\sum_{j=1}^\infty
\gamma_j < \infty$.

As a direct consequence of Lemma \ref{lem:sum-r} we obtain a bound on the
number of $\bsh\in\ZZ^d$ such that $r_{\beta,\bsgamma}(\bsh)$ is small.
\begin{corollary}\label{cor:number}
Let $d\in\NN$, $\beta>1$, $T>0$, and $\bsgamma\in(0,1]^\NN$, and define
\[
\calA_{\beta,\bsgamma}(T)
\,:=\, \{\bsh\in\ZZ^d\colon\, r_{\beta,\bsgamma}(\bsh) \,\le\, T \}.
\]
Then we have
\[
\abs{\calA_{\beta,\bsgamma}(T)} \,\le\, T\, V_d(\beta,\bsgamma).
\]
\end{corollary}
\begin{proof}
From Lemma~\ref{lem:sum-r} we obtain
\[
V_d(\beta,\bsgamma) \,\ge\, \sum_{\bsh\in\calA_{\beta,\bsgamma}(T)} \frac{1}{r_{\beta,\bsgamma}(\bsh)}
\,\ge\, \frac{|\calA_{\beta,\bsgamma}(T)|}{T}.
\]
This proves the result.
\end{proof}

\medskip

The next lemma is useful in bounding the number of points in the dual lattice
for a given generating vector.

\begin{lemma}\label{lem:divisor}
For every prime number $p$, every $d\in\NN$, and $\bsh\in\ZZ^d$ we have
\[
\#\bigl\{\bsz\in \{1:p-1\}^d\colon\,\bsh\cdot\bsz \equiv_p 0\bigl\}\;\le\begin{cases}
(p-1)^{d} & \text{ if }\, \bsh\equiv_p\bszero,\\
(p-1)^{d-1} & \text{ otherwise,}
\end{cases}
\]
with equality in the first case,
where by $\bsh\equiv_p\bszero$ we mean that each component of $\bsh$ is
congruent to zero modulo $p$.
\end{lemma}

\begin{proof}
Recall that $\bsh\cdot\bsz=\sum_{j=1}^d h_j z_{j}$. If
$\bsh\equiv_p\bszero$, then every $\bsz\in\{1:p-1\}^d$ is a solution to
$\bsh\cdot\bsz \equiv_p 0$. Hence, there are $(p-1)^{d}$ solutions in this
case. Otherwise, there is an $\ell\in\{1:d\}$ such that
$h_\ell\not\equiv_p0$. Now, for any fixed $z_j$, $j\neq\ell$, we have
$\bsh\cdot\bsz\equiv_p0$ if and only if $h_\ell z_\ell \equiv_p
-\sum_{j\neq\ell} h_j z_{j}$. Since $h_\ell\not\equiv_p0$ and $p$ is prime
there is at most one such $z_\ell\in\{1:p-1\}$. This shows that there are
no more than $(p-1)^{d-1}$ solutions to $\bsh\cdot\bsz \equiv_p 0$ in this
case.
\end{proof}

A similar result to the following proposition can be found in many papers,
see, e.g., \cite[Lemma~2]{SW01}.

\begin{proposition}\label{prop:av-z}
Let $p$ be prime, $d\in\NN$, and $\bsgamma\in(0,1]^\NN$.
For any $\beta>1$, we have
\[
  \frac{1}{(p-1)^d} \sum_{\bsz\in\{1:p-1\}^d} P_{\beta,\bsgamma}(p,\bsz)
	\;\le\; \frac{V_d(\beta,\bsgamma)}{p}\,.
\]
\end{proposition}

\begin{proof}
By definition we have
\begin{align*}
\frac{1}{(p-1)^d} \sum_{\bsz\in\{1:p-1\}^d} P_{\beta,\bsgamma}(p,\bsz)
 &\,=\,
 \frac{1}{(p-1)^d}\sum_{\bsz\in\{1:p-1\}^d}
 \sum_{\substack{\bsh\in\ZZ^d\setminus \{\bszero\}\\ \bsh\cdot\bsz\equiv_p 0}} \frac{1}{r_{\beta,\bsgamma}(\bsh)} \\
 &\,=\, \frac{1}{(p-1)^d}\sum_{\bsh\in\ZZ^d\setminus\{\bszero\}}
  \frac{\#\bigl\{\bsz\in\{1:p-1\}^d\colon\,\bsh\cdot\bsz \equiv_p 0\bigl\}}{r_{\beta,\bsgamma}(\bsh)}\,.
\end{align*}
We split this sum into a sum over all $\bsh$ with $\bsh\equiv_p \bszero$ and the
remaining $\bsh$, and bound the sums separately.
By Lemma~\ref{lem:divisor} we have for the sum over $\bsh\equiv_p \bszero$,
\begin{align*}
 \frac{1}{(p-1)^d}\sum_{\substack{\bsh\in\ZZ^d\setminus\{\bszero\}\\ \bsh\equiv_p \bszero}}
 \frac{\#\bigl\{\bsz\in\{1:p-1\}^d\colon\,\bsh\cdot\bsz \equiv_p 0\bigl\}}{r_{\beta,\bsgamma}(\bsh)}
 &\,=\,
 \sum_{\substack{\bsh\in\ZZ^d\setminus\{\bszero\}\\ \bsh\equiv_p \bszero}}
 \frac{1}{r_{\beta,\bsgamma}(\bsh)}\\
 &\,=\,
 \sum_{\bsl\in\ZZ^d\setminus\{\bszero\}} \frac{1}{r_{\beta,\bsgamma}(p\, \bsl)}\\
 &\,\le\, \frac{1}{p^{\beta}} \sum_{\bsl\in\ZZ^d\setminus\{\bszero\}}
	\frac{1}{r_{\beta,\bsgamma}(\bsl)}\\
 &\,\le\, \frac{1}{3p}\, V_d(\beta,\bsgamma).
\end{align*}
The penultimate inequality easily follows from the definition of
$r_{\beta,\bsgamma}(\bsh)$ and $\bsh\neq\bszero$,
and the last inequality follows from Lemma~\ref{lem:sum-r}.
For the sum over $\bsh\not\equiv_p \bszero$ we use again
Lemma~\ref{lem:sum-r} and Lemma~\ref{lem:divisor}, and obtain
\[\begin{split}
 \frac{1}{(p-1)^d}\sum_{\substack{\bsh\in\ZZ^d\setminus\{\bszero\}\\ \bsh\not\equiv_p \bszero}}
 \frac{\#\bigl\{\bsz\in\{1:p-1\}^d\colon\,\bsh\cdot\bsz \equiv_p 0\bigl\}}{r_{\beta,\bsgamma}(\bsh)}
 \,&\le\, \frac{1}{p-1}\sum_{\bsh\in\ZZ^d} \frac{1}{r_{\beta,\bsgamma}(\bsh)} \\
 \,&=\, \frac{1}{3(p-1)}\, V_d(\beta,\bsgamma).
\end{split}\]
Combining the last two estimates with $1/(p-1)\le2/p$ leads to the
desired result.
\end{proof}

\bigskip

Clearly, there must be one choice of $\bsz$ such that
$P_{\beta,\bsgamma}(p,\bsz)$ is as good as the average. By this argument,
and using \eqref{eq:error-det}, it is typically shown that there exists a
generating vector that makes the worst case error small. Although this is
clearly not a constructive argument, there are methods to generate such
vectors in an efficient way, in particular component-by-component
algorithms. These constructions, dating back to Korobov, were re-invented
in 2002 in \cite{SR02}, and proven to yield optimal results in \cite{K03}.
Moreover, there exists a fast variant due to \cite{NC06} which is heavily
used nowadays.

For the upcoming analysis, it is not enough to have a single ``good''
generating vector. However, as the next corollary shows, there are
actually many of them.
\begin{corollary} \label{cor:good-z}
Let $p$ be prime, $d\in\NN$, and $\bsgamma\in(0,1]^\NN$. For any $\beta>1$
and $\tau\in (0,1)$, there exist at least $\lceil \tau(p-1)^d \rceil$
generating vectors $\bsz\in\{1:p-1\}^d$ such that
\begin{equation*} \label{eq:good-z}
  P_{\beta,\bsgamma}(p,\bsz)
  \,\le\, \frac{1}{1-\tau}\cdot\frac{V_d(\beta,\bsgamma)}{p}\,.
\end{equation*}
\end{corollary}
\begin{proof}
Suppose to the contrary that $P_{\beta,\bsgamma}(p,\bsz)$ is smaller than
or equal to the right-hand side above for $\chi$ choices of $\bsz$, where
$\chi\le\lceil \tau(p-1)^d \rceil-1 < \tau(p-1)^d$. The number of $\bsz$
such that $P_{\beta,\bsgamma}(p,\bsz)$ is larger, i.e., $(p-1)^d-\chi$,
therefore satisfies $(p-1)^d-\chi > (1-\tau)(p-1)^d$. Then the average
over all $\bsz$ is bigger than $V_d(\beta,\bsgamma)/p$, which contradicts
Proposition~\ref{prop:av-z}.
\end{proof}

Most of the results discussed so far in this section hold for arbitrary
$\beta>1$. However, for the further analysis we need a bound on
$\rho_{\alpha,\bsgamma}(p,\bsz)$ for many $\bsz$ for all $\alpha>0$. It is
known that the classical Zaremba index, corresponding to $\alpha=1$ and
$\bsgamma =\bsone$, has a growth behavior like $p$ (with additional
logarithmic terms), see, e.g.,~\cite{N92}, so we expect a growth behavior
of $\rho_{\alpha,\bsgamma}(p,\bsz)$ comparable to $p^\alpha$. This is made
more precise in the following lemma.

\begin{lemma} \label{lem:good-z}
Let $p$ be prime, $d\in\NN$, $\alpha>0$, and $\bsgamma\in(0,1]^\NN$. For
any $\tau\in (0,1)$, there exist at least $\lceil \tau(p-1)^d \rceil$
generating vectors $\bsz\in \{1:p-1\}^d$ such that
\begin{equation} \label{eq:good-Zaremba}
 \rho_{\alpha,\bsgamma}(p,\bsz) \,\ge\,
  \left(\frac{(1-\tau)\,p}{V_d\bigl({\alpha}/{\lambda},\bsgamma^{1/\lambda}\bigr)}\right)^{\lambda}
  \quad\mbox{for all } \lambda\in (0,\alpha)\,.
\end{equation}
\end{lemma}

\begin{proof}
Using \eqref{eq:Zaremba} and
$r_{\alpha,\bsgamma}(\bsh)=(r_{\alpha/\lambda,\bsgamma^{1/\lambda}}(\bsh))^{\lambda}$
for all $\lambda\in(0,\alpha)$, we obtain for any $\bsz\in\{1:p-1\}^d$
that
\[
  \frac{1}{\rho_{\alpha,\bsgamma}(p,\bsz)}
\,=\, \frac{1}{(\rho_{\alpha/\lambda,\bsgamma^{1/\lambda}}(p,\bsz))^{\lambda}}
\,<\, (P_{\alpha/\lambda,\bsgamma^{1/\lambda}}(p,\bsz))^{\lambda}.
\]
Now for each $\lambda\in (0,\alpha)$, we know from
Corollary~\ref{cor:good-z} that the inequality
\begin{equation} \label{eq:extra}
     (P_{\alpha/\lambda,\bsgamma^{1/\lambda}}(p,\bsz))^{\lambda}
\,\le\, \left(\frac{V_d(\alpha/\lambda,\bsgamma^{1/\lambda})}{(1-\tau)p}\right)^\lambda
\end{equation}
holds for at least $\lceil\tau(p-1)^d \rceil$ different generating vectors
$\bsz =\bsz(\lambda)$. To obtain the set of generating vectors that
satisfies~\eqref{eq:extra} for all $\lambda\in (0,\alpha)$, we consider
the infimum of the right-hand side over~$\lambda$. If this infimum is
attained for some $\lambda^*\in (0,\alpha)$, then those generating vectors
$\bsz(\lambda^*)$ which correspond to $\lambda^*$ will satisfy the desired
bound \eqref{eq:good-Zaremba} for all other values of $\lambda$. If it is
not attained in the interval $(0,\alpha)$, then since the limit as
$\lambda\to\alpha$ is clearly infinity, the infimum is the limit as
$\lambda\to0$. However, this limit equals 1, which makes the statement of
the lemma trivial.
\end{proof}

\begin{remark}
For a given prime $p$, let us choose a generating vector
$\bsz\in\{1:p-1\}^d$ uniformly at random. By Lemma~\ref{lem:good-z} we
have that the probability of obtaining, by this procedure, a ``good''
generating vector, i.e., a vector $\bsz$ that satisfies the
bound~\eqref{eq:good-Zaremba}, is at least $\tau$. Besides possible
implementation issues, we did not find any advantage of adjusting $\tau$.
Hence, to ease the notation, we set $\tau=1/2$ in the following.

\end{remark}

\goodbreak

\section{New results in the randomized setting} \label{sec:rand}

\subsection{The randomized lattice algorithm without shift}

Let $n\in\NN$, $n\ge 2$, and
\[
  \calP_n \,:=\, \{ p\colon\; p \mbox{ is prime and }
	\; \lceil n/2 \rceil +1 \,\le\, p \,\le\, n\}.
\]
Let $d\in\NN$, $\alpha>0$, and $\bsgamma\in(0,1]^\NN$. For each
$p\in\calP_n$, let $\calZ_p$ denote the set of good generating vectors
$\bsz$ in the sense of Lemma~\ref{lem:good-z}, with $\tau=1/2$, that is,
\begin{equation} \label{eq:Zp-def}
  \calZ_p \,:=\, \calZ_{p,\alpha,\bsgamma} \,:=\,
  \left\{ \bsz\in\{1:p-1\}^d\colon\; \rho_{\alpha,\bsgamma}(p,\bsz) \,\ge\,
  \left(\frac{p}{2\,V_d\bigl({\alpha}/{\lambda},\bsgamma^{1/\lambda}\bigr)}\right)^{\lambda}
  \mbox{ for all } \lambda\in (0,\alpha)\right\}.
\end{equation}
We know from Lemma~\ref{lem:good-z} that $|\calZ_p| \,\ge\,
\lceil(p-1)^d/2\rceil$.

Our random algorithm $M_n$ is defined by first randomly and uniformly
selecting a prime $p\in\calP_n$ and then randomly and uniformly selecting
a generating vector $\bsz\in\calZ_p$. The randomized error
\eqref{eq:rand-err} for the lattice algorithm $M_n$ is then given
precisely by
\[
e^{\rm ran}_{d,\alpha,\bsgamma}(M_n) \,=\,
 \sup_{\substack{f\in \calH_{d,\alpha,\bsgamma}\\ \|f\|_{d,\alpha,\bsgamma}\le 1}} \left(
 \frac{1}{|\calP_n|} \sum_{p\in\calP_n} \frac{1}{|\calZ_p|} \sum_{\bsz\in\calZ_p}
 |Q_{d,p,\bsz} (f)- I_d(f)| \right).
\]
We stress that the randomized error is \textit{not} the same as the
average of the worst case errors \eqref{eq:wor-err} of a set of
deterministic lattice rules. (Here the averaging occurs inside the
supremum rather than outside.)

\begin{theorem}\label{thm1}
For $\alpha>1/2$, $\lambda\in (1/2,\alpha)$,
$\bsgamma\in(0,1]^\NN$, and
\begin{equation} \label{eq:suff-large}
 n \,\ge\, 4\,V_d\bigl({\alpha}/{\lambda},\bsgamma^{1/\lambda}\bigr),
\end{equation}
the randomized error of the randomized lattice algorithm $M_n$ satisfies
\[
e^{\rm ran}_{d,\alpha,\bsgamma}(M_n)
\;\le\; C_{\lambda,\delta}\, \big[V_d\bigl({\alpha}/{\lambda},\bsgamma^{1/\lambda}\bigr)\big]^{\lambda}\;
n^{-\lambda-1/2+\delta}
\]
for arbitrary $\delta$ satisfying $0 < \delta <
\min(\lambda-1/2,1)$, where $V_d$ is defined as in \eqref{eq:V} and
$C_{\lambda,\delta}$ is a constant depending only on $\lambda$ and
$\delta$. The upper bound is independent of $d$ if $\sum_{j=1}^\infty
\gamma_j^{1/\lambda}<\infty$.
\end{theorem}

\begin{remark} \label{rem:imp}
A simple probabilistic approach to implement this algorithm is to first
randomly select a prime $p\in\calP_n$, and then randomly select a vector
$\bsz\in \{1:p-1\}^d$ repeatedly until the condition in \eqref{eq:Zp-def}
is satisfied. The chance of getting a vector $\bsz\in\calZ_p$ after $k$
tries is $1-\tau^k = 1 - 2^{-k}$.

Note that the inequality in \eqref{eq:Zp-def} should hold simultaneously
for all values of $\lambda\in (0,\alpha)$, and so in a practical
implementation it is not easy to verify whether a given vector $\bsz$
belongs to~$\calZ_p$. An easy way around this is to relax the definition
of $\calZ_p$ to a set of generating vectors $\calZ_{p,\lambda}$ which
assumes the given bound only for a fixed parameter $\lambda\in
(0,\alpha)$, and consider the corresponding random algorithm
$M_{n,\lambda}$ which depends on the parameter $\lambda$. In this case,
Theorem~\ref{thm1} holds for $M_{n,\lambda}$ for this parameter $\lambda$,
and the implementation of the algorithm is straightforward and simple.

We leave it for future research if sampling from the set $\calZ_p$ itself
can be done efficiently, e.g., by a component-by-component-type algorithm.
\end{remark}

\begin{proofof}{Theorem~\ref{thm1}}
We first define
\begin{equation}\label{eq:B}
B_n \,:=\, B_{n,\alpha,\bsgamma,\lambda} \,:=\,
\left(\frac{n}{4\,V_d\bigl({\alpha}/{\lambda},\bsgamma^{1/\lambda}\bigr)}\right)^{\lambda}.
\end{equation}
Then the condition \eqref{eq:suff-large} ensures that $B_n\ge 1$.

It follows from the definition of $\calZ_p$ in \eqref{eq:Zp-def} that for
all $p\in\calP_n$ and $\bsz\in\calZ_p$ we have $p>n/2$ and
$\rho_{\alpha,\bsgamma}(p,\bsz)>B_n$; and consequently for every
$\bsh\in\ZZ^d \setminus\{\bszero\}$ with $\bsh\cdot\bsz\equiv_p 0$ we have
$r_{\alpha,\bsgamma}(\bsh)>B_n$. From \eqref{eq:lattice-error} we obtain
\begin{align*}
 \frac{1}{|\calP_n|} \sum_{p\in\calP_n} \frac{1}{|\calZ_p|} \sum_{\bsz\in\calZ_p}
 |Q_{d,p,\bsz} (f)-I_d(f)|
 &\,=\, \frac{1}{|\calP_n|} \sum_{p\in\calP_n} \frac{1}{|\calZ_p|} \sum_{\bsz\in\calZ_p}
 \Bigg|
 \sum_{\substack{\bsh\in\ZZ^d\setminus\{\bszero\}\\ \bsh\cdot\bsz\equiv_p 0}}
 \hat{f} (\bsh)
 \Bigg| \\
 &\,\le\, \frac{1}{|\calP_n|} \sum_{p\in\calP_n} \frac{1}{|\calZ_p|} \sum_{\bsz\in\calZ_p}
 \sum_{\substack{\bsh\in\ZZ^d\setminus\{\bszero\}\\ \bsh\cdot\bsz\equiv_p 0}}
 |\hat{f} (\bsh)| \\
 &\, =\,
 \sum_{\substack{\bsh\in\ZZ^d\setminus\{\bszero\}\\ r_{\alpha,\bsgamma}(\bsh)> B_{n}}}
 \omega_n (\bsh)\, |\hat{f} (\bsh)| \\
 &\,\le\,
 \Bigg(\sum_{\substack{\bsh\in\ZZ^d\setminus\{\bszero\}\\
 r_{\alpha,\bsgamma}(\bsh)> B_n}}
 \left(\frac{\omega_n (\bsh)}{r_{\alpha,\bsgamma}(\bsh)}\right)^2\Bigg)^{1/2}\, \|f\|_{\calH_{d,\alpha,\bsgamma}},
\end{align*}
where
\begin{equation}\label{eq:omega}
\omega_n(\bsh) \,:=\, \omega_{n,\alpha,\bsgamma,\lambda}(\bsh) \,:=\,
  \frac{1}{|\calP_n|} \sum_{p\in\calP_n} \frac{1}{|\calZ_p|} \sum_{\bsz\in\calZ_p}
  \bbone(\bsh\cdot\bsz\equiv_p 0)\,,
\end{equation}
with $\bbone(\cdot)$ denoting the indicator function. The last inequality
was obtained by multiplying and dividing the penultimate expression by
$r_{\alpha,\bsgamma}(\bsh)$ and then applying the Cauchy--Schwarz
inequality. Hence we conclude that
\[
e^{\rm ran}_{d,\alpha,\bsgamma}(M_n)
 \,\le\, \Bigg(\sum_{\substack{\bsh\in\ZZ^d\setminus\{\bszero\}\\
 r_{\alpha,\bsgamma}(\bsh)> B_n}} \left(\frac{\omega_n (\bsh)}{r_{\alpha,\bsgamma}(\bsh)}\right)^2\Bigg)^{1/2}.
\]

We now proceed to obtain a bound on $\omega_n(\bsh)$. For fixed $p$, if
$\bsh\equiv_p \bszero$ then $\bsh\cdot\bsz\equiv_p 0$ holds for all
$\bsz\in\calZ_p$. On the other hand, if $\bsh\not\equiv_p \bszero$ then we
may bound the last sum in \eqref{eq:omega} by the number of all
$\bsz\in\{1:p-1\}^d$ with $\bsh\cdot\bsz\equiv_p 0$. We already computed
this number in Lemma~\ref{lem:divisor}. Thus we have
\[
  \begin{cases}
  \displaystyle
  \frac{1}{|\calZ_p|} \sum_{\bsz\in\calZ_p} \bbone(\bsh\cdot\bsz\equiv_p 0)
  \,=\, 1 & \mbox{if $\bsh\equiv_p \bszero$}, \\
  \displaystyle
  \frac{1}{|\calZ_p|} \sum_{\bsz\in\calZ_p} \bbone(\bsh\cdot\bsz\equiv_p 0)
  \,\le\, \frac{(p-1)^{d-1}}{|\calZ_p|} \,\le\, \frac{(p-1)^{d-1}}{\lceil(p-1)^d/2\rceil}
	\,\le\, \frac{4}{n}
  & \mbox{if $\bsh\not\equiv_p \bszero$},
  \end{cases}
\]
and therefore
\[
  \omega_n(\bsh) \,\le\,
  \frac{1}{|\calP_n|} \sum_{p\in\calP_n}
  \left(
  \bbone(\bsh\equiv_p \bszero)  + \frac{4\,\bbone(\bsh\not\equiv_p \bszero)}{n}
  \right)
  \,\le\,
  \frac{1}{|\calP_n|} \sum_{p\in\calP_n}
  \bbone(\bsh\equiv_p \bszero)  + \frac{4}{n}\,.
\]
Note that any number $h\in\bbN$ has at most $\log_M(h)$ prime divisors
greater than $M\in\bbN$. So for $\bsh\ne\bszero$ the number of primes
$p\ge\lceil n/2\rceil+1$ for
which $\bsh\equiv_p \bszero$ holds is at most $\log_{\lceil n/2\rceil+1}(|\bsh|_\infty)$,
that is,
\[
   \sum_{p\in\calP_n} \bbone(\bsh\equiv_p \bszero) \,\le\, \log_{\lceil n/2\rceil+1}(|\bsh|_\infty)
   \,=\, \frac{\log(|\bsh|_\infty)}{\log(\lceil n/2\rceil+1)}
	\,\le\, \frac{2 \log(|\bsh|_\infty)}{\log(n)}
\]
for all $n\ge2$. Combining this with the estimate $|\calP_n| > c'
n/\log(n)$ for some $c'>0$, see, e.g., \cite{RS62}, we conclude that
\begin{equation}\label{eq:omega-bound}
\omega_n(\bsh)
  \,\le\, \frac{2}{c'\,n} \log(|\bsh|_\infty)  + \frac{4}{n}
  \,\le\, \frac{c}{n} \log(1+|\bsh|_\infty)
\end{equation}
for some $c<\infty$,
which yields
\[
e^{\rm ran}_{d,\alpha,\bsgamma}(M_n)
 \,\le\, \frac{c}{n} \Bigg(\sum_{\substack{\bsh\in\ZZ^d\setminus\{\bszero\}\\
 r_{\alpha,\bsgamma}(\bsh)> B_n}} \frac{\log^2(1+|\bsh|_\infty)}{(r_{\alpha,\bsgamma}(\bsh))^2}\Bigg)^{1/2}
 \,=\, \frac{c}{n} \Bigg(\sum_{\substack{\bsh\in\ZZ^d\setminus\{\bszero\}\\
 r_{\beta,\bsgamma'}(\bsh)> B_n^{1/\lambda}}} \frac{\log^2(1+|\bsh|_\infty)}{(r_{\beta,\bsgamma'}(\bsh))^{2\lambda}}\Bigg)^{1/2}
\]
with $\beta:=\alpha/\lambda>1$ and $\bsgamma'=\bsgamma^{1/\lambda}$. Here
we used
$r_{\alpha,\bsgamma}(\bsh)=(r_{\alpha/\lambda,\bsgamma^{1/\lambda}}(\bsh))^\lambda$.
Now we apply $|\bsh|_\infty \le r_{\beta,\bsgamma'}(\bsh)$ since $\beta >
1$ and $\bsgamma \in (0,1]^\NN$, and $\log(1 + x)\le x^\delta/\delta$ which
holds for all $\delta\in(0,1]$ and $x>0$, and obtain
\begin{equation}\label{eq:error-temp}
e^{\rm ran}_{d,\alpha,\bsgamma}(M_n)
 \,\le\, \frac{c}{\delta n} \Bigg(\sum_{\substack{\bsh\in\ZZ^d\setminus\{\bszero\}\\
 r_{\beta,\bsgamma'}(\bsh)> B_n^{1/\lambda}}} \frac{1}{(r_{\beta,\bsgamma'}(\bsh))^{2(\lambda-\delta)}}\Bigg)^{1/2},
\end{equation}
where we further restrict $\delta$ to be such that $2(\lambda -
\delta)
> 1$, which is equivalent to $\delta < \lambda - 1/2$.

To complete the proof we will show a suitable upper bound on
\[
\sum_{\substack{\bsh\in\ZZ^d\setminus\{\bszero\}\\
 r_{\beta,\bsgamma'}(\bsh)> B_n^{1/\lambda}}} \frac{1}{(r_{\beta,\bsgamma'}(\bsh))^{2(\lambda-\delta)}}.
\]
For $T>0$, we define the set
\[
  \calA(T) \,:=\, \calA_{\beta,\bsgamma'}(T) \,=\, \{\bsh \in\ZZ^d\colon r_{\beta,\bsgamma'}(\bsh)\le T\}.
\]
Recall from Corollary~\ref{cor:number} that $\abs{\calA(T)}\le T\,
V_d(\alpha/\lambda,\bsgamma^{1/\lambda})$. Moreover, we have for $u>1$ and
$T\ge1$,
\[
 T^{-u}- (T+1)^{-u} \,=\, u \int_T^{T+1} x^{-u-1} \rd x
 \,\le\, u\, T^{-u-1}.
\]
Recall also that $B_n^{1/\lambda} \ge 1$. Since $\bsh\notin\calA(T)$
implies $1/r_{\beta,\bsgamma'}(\bsh)< 1/T$, we obtain
\begin{align*}
 &\sum_{\substack{\bsh\in\ZZ^d\setminus\{\bszero\}\\
 r_{\beta,\bsgamma'}(\bsh)> B_n^{1/\lambda}}} \frac{1}{(r_{\beta,\bsgamma'}(\bsh))^{2(\lambda-\delta)}}
 \,\le\, \sum_{T=\lfloor B_n^{1/\lambda}\rfloor}^\infty \,
 \sum_{\bsh\in\calA(T+1)\setminus\calA(T)} \frac{1}{(r_{\beta,\bsgamma'}(\bsh))^{2(\lambda-\delta)}} \\
 &\,\le\,
 \sum_{T=\lfloor B_n^{1/\lambda}\rfloor}^\infty \,
 T^{-2(\lambda-\delta)}\, \Bigl(\abs{\calA(T+1)}-\abs{\calA(T)}\Bigr) \\
 &\,\le\, \sum_{T=\lfloor B_n^{1/\lambda}\rfloor}^\infty \,
 \Bigl(T^{-2(\lambda-\delta)}-(T+1)^{-2(\lambda-\delta)}\Bigr)\, \abs{\calA(T+1)} \\
 &\,\le\, 2(\lambda-\delta)\,\sum_{T=\lfloor B_n^{1/\lambda}\rfloor}^\infty \,
 T^{-2(\lambda-\delta)-1}\, \abs{\calA(T+1)} \\
 &\,\le\, 4(\lambda-\delta)\, V_d(\alpha/\lambda,\bsgamma^{1/\lambda})\,
 \sum_{T=\lfloor B_n^{1/\lambda}\rfloor}^\infty \, T^{-2(\lambda-\delta)} \\
 &\,\le\, 4(\lambda-\delta)\, V_d(\alpha/\lambda,\bsgamma^{1/\lambda})\,
 \bigg( \left(\lfloor B_n^{1/\lambda}\rfloor\right)^{-2(\lambda-\delta)}  +
 \int_{\lfloor B_n^{1/\lambda}\rfloor}^\infty \, x^{-2(\lambda-\delta)}\rd x \bigg)\\
 &\,\le\, 4(\lambda-\delta)\, V_d(\alpha/\lambda,\bsgamma^{1/\lambda})\,
 \left(1 + \frac{1}{2(\lambda-\delta) -1}\right)
 \left(\lfloor B_n^{1/\lambda} \rfloor \right)^{-2(\lambda-\delta)+1}\\
 &\,\le\, 4(\lambda-\delta)\, V_d(\alpha/\lambda,\bsgamma^{1/\lambda})\,
 \left(1 + \frac{1}{2(\lambda-\delta) -1}\right)
 2^{2(\lambda-\delta)-1}\left( B_n^{1/\lambda}  \right)^{-2(\lambda-\delta)+1} \\
 &\,=\, \frac{2^{2(\lambda-\delta)+2}(\lambda-\delta)^2}{2(\lambda-\delta)-1}\, V_d(\alpha/\lambda,\bsgamma^{1/\lambda})\,
 \left( B_n^{1/\lambda}  \right)^{-2(\lambda-\delta)+1},
\end{align*}
where we used $\lfloor x \rfloor \ge x/2$ for $x\ge 1$ in the second to
last inequality.
Together with \eqref{eq:B} and \eqref{eq:error-temp}, this shows
\begin{align*}
e^{\rm ran}_{d,\alpha,\bsgamma}(M_n)
 &\,\le\, \frac{c}{\delta n}\; \left(
 \frac{2^{2(\lambda-\delta)+2}(\lambda-\delta)^2}{2(\lambda-\delta)-1}
 V_d(\alpha/\lambda,\bsgamma^{1/\lambda})\,
\left(\frac{4\,V_d\bigl({\alpha}/{\lambda},\bsgamma^{1/\lambda}\bigr)}{n}\right)^{2(\lambda-\delta)-1}\right)^{1/2} \\
 &\,\le\,
 \frac{c\,2^{3(\lambda-\delta)}(\lambda-\delta)}{\delta\sqrt{2(\lambda-\delta)-1}}\;
[V_d\bigl({\alpha}/{\lambda},\bsgamma^{1/\lambda}\bigr)]^{\lambda}\;
n^{-\lambda-1/2+\delta}.
\end{align*}
Note that the coefficient of $n^{-\lambda-1/2+\delta}$ in the latter term
is bounded under the assumptions we made, i.e., $\alpha>\lambda>1/2$ and
$0<\delta<\lambda-1/2$, as well as
$V_d(\alpha/\lambda,\bsgamma^{1/\lambda}) < \infty$ independently of $d$
if $\sum_{j=1}^\infty \gamma_j^{1/\lambda} < \infty$, see~\eqref{eq:V}.
This completes the proof.
\end{proofof}

\subsection{The randomized lattice algorithm with shift}

We are now going to prove in Theorem~\ref{thm2} below that
$\widetilde{M}_n$, i.e., the random lattice rule with an additional shift,
satisfies a similar upper bound on its randomized error as $M_n$, but also
for functions from $\calH_{d,\alpha,\bsgamma}$ with $0<\alpha\le1/2$. Note
that, by allowing general $\alpha>0$ in Theorem \ref{thm2}, we are
considering a larger function class than in Theorem \ref{thm1}.

The randomized error \eqref{eq:rand-err} for the algorithm
$\widetilde{M}_n$ is given by
\[ e^{\rm ran}_{d,\alpha,\bsgamma}\bigl(\widetilde{M}_n\bigr)  \,=\,
 \sup_{\substack{f\in\calH_{d,\alpha,\bsgamma}\\ \|f\|_{d,\alpha,\bsgamma}\le 1}}
 \left(
 \frac{1}{|\calP_n|} \sum_{p\in\calP_n} \frac{1}{|\calZ_p|} \sum_{\bsz\in\calZ_p}
 \EE_{\bsU}\Big[\left|Q_{d,p,\bsz} \bigl(f(\{\cdot+\bsU\})\bigr) - I_d(f)\right|\Big] \right),
\]
where $\EE_{\bsU}$ denotes expectation with respect to the random shift
$\bsU\in [0,1]^d$. For this algorithm we can even bound the
\emph{root-mean-square error} (or standard deviation)
\begin{align*}
 e^{\rm rms}_{d,\alpha,\bsgamma}\bigl(\widetilde{M}_n\bigr)
\,:=&\, \sup_{\substack{f\in\calH_{d,\alpha,\bsgamma}\\ \|f\|_{d,\alpha,\bsgamma}\le 1}}
\sqrt{\EE\left[\left| \widetilde{M}_n(f) - I_d(f) \right|^2\right]} \\
\,=&\,
\sup_{\substack{f\in\calH_{d,\alpha,\bsgamma}\\ \|f\|_{d,\alpha,\bsgamma}\le 1}}
 \left(
 \frac{1}{|\calP_n|} \sum_{p\in\calP_n} \frac{1}{|\calZ_p|} \sum_{\bsz\in\calZ_p}
 \EE_{\bsU}\left[\left|Q_{d,p,\bsz} \bigl(f(\{\cdot+\bsU\})\bigr) - I_d(f)\right|^2\right] \right)^{1/2}.
\end{align*}
Clearly, $e^{\rm ran}_{d,\alpha,\bsgamma}\bigl(\widetilde{M}_n\bigr) \le
e^{\rm rms}_{d,\alpha,\bsgamma}\bigl(\widetilde{M}_n\bigr)$.

\medskip

\begin{theorem}\label{thm2}
For $\alpha>0$, $\lambda\in (0,\alpha)$,
and $\bsgamma \in (0,1]^\NN$,
the randomized worst case error of the randomized lattice algorithm with
shift $\widetilde{M}_n$ satisfies
\[
 e^{\rm ran}_{d,\alpha,\bsgamma}\bigl(\widetilde{M}_n\bigr)
\,\le\, e^{\rm rms}_{d,\alpha,\bsgamma}\bigl(\widetilde{M}_n\bigr)
\,\le\, \left(\frac{c}{\alpha\delta}\right)^{1/2}
\left(4\,V_d\bigl({\alpha}/{\lambda},\bsgamma^{1/\lambda}\bigr)\right)^{\lambda} \;
n^{-\lambda-1/2+\delta \lambda/2}
\]
for arbitrary $\delta$ satisfying $0 < \delta < \min(1/\alpha,2)$,
where again $V_d$ is defined as in \eqref{eq:V} and $c$ is an absolute
constant. The upper bound is independent of $d$ if $\sum_{j=1}^\infty
\gamma_j^{1/\lambda}<\infty$.
\end{theorem}

\medskip

\begin{proof}
By a variation of Poisson's summation formula, we have
\[
\EE_{\bsU}\left[\left|Q_{d,p,\bsz} \bigl(f(\{\cdot+\bsU\})\bigr) - I_d(f)\right|^2\right]
\,=\, \EE_{\bsU} \Bigg[  \Bigl|
	\sum_{\satop{\bsh\in\bbZ^d\setminus\{\bszero\}}{\bsh\cdot\bsz \equiv_p 0}}
  \hat{f}(\bsh) e^{2\pi i\, \bsh \cdot\bsU}\Bigr|^2 \Bigg]
\,=\, \sum_{\satop{\bsh\in\bbZ^d\setminus\{\bszero\}}{\bsh\cdot\bsz \equiv_p 0}}
  \bigl|\hat{f}(\bsh)\bigr|^2,
\]
and therefore
\begin{align*}
 \big(e^{\rm rms}_{d,\alpha,\bsgamma}\bigl(\widetilde{M}_n\bigr)\big)^2
 &\,=\, \sup_{\substack{f\in\calH_{d,\alpha,\bsgamma}\\ \|f\|_{d,\alpha,\bsgamma}\le 1}}\,
	\frac{1}{|\calP_n|} \sum_{p\in\calP_n} \frac{1}{|\calZ_p|} \sum_{\bsz\in\calZ_p}
 \sum_{\substack{\bsh\in\ZZ^d\setminus\{\bszero\}\\ \bsh\cdot\bsz\equiv_p 0}}
 \bigl|\hat{f} (\bsh)\bigr|^2 \\
 &\,=\, \sup_{\substack{f\in\calH_{d,\alpha,\bsgamma}\\ \|f\|_{d,\alpha,\bsgamma}\le 1}}
 \sum_{\substack{\bsh\in\ZZ^d\setminus\{\bszero\}\\ r_{\alpha,\bsgamma}(\bsh)> B_{n}}}
 \omega_n (\bsh)\, \bigl|\hat{f} (\bsh)\bigr|^2 \\
 &\,\le\,
\sup_{\substack{\bsh\in\ZZ^d\setminus\{\bszero\}\\
 r_{\alpha,\bsgamma}(\bsh)> B_n}}
 \frac{\omega_n (\bsh)}{(r_{\alpha,\bsgamma}(\bsh))^2}
\end{align*}
with $\omega_n (\bsh)$ from \eqref{eq:omega}.
{}From~\eqref{eq:omega-bound} we know that
$$
\omega_n(\bsh)\le c\, \log(1+|\bsh|_\infty)/n.
$$
We apply $(|\bsh|_\infty)^\alpha \le r_{\alpha,\bsgamma}(\bsh)$ and
$\log(1+x)\le x^{\alpha\delta}/(\alpha\delta)$ for all
$\delta\in(0,1/\alpha]$ to obtain
\[
\big(e^{\rm rms}_{d,\alpha,\bsgamma}\bigl(\widetilde{M}_n\bigr)\big)^2
\,\le\, \frac{c}{\alpha\delta n} \sup_{\substack{\bsh\in\ZZ^d\setminus\{\bszero\}\\
 r_{\alpha,\bsgamma}(\bsh)> B_n}}
 \frac{1}{(r_{\alpha,\bsgamma}(\bsh))^{2-\delta}}
\,\le\, \frac{c}{\alpha\delta n}\, B_n^{\delta-2},
\]
where we further restrict $\delta$ to be such that $\delta<2$.
Together with \eqref{eq:B} this yields
\begin{eqnarray*}
e^{\rm rms}_{d,\alpha,\bsgamma}\bigl(\widetilde{M}_n\bigr)
 \,&\le&\, \left(\frac{c}{\alpha\delta}\right)^{1/2}
	\left(4\,V_d\bigl({\alpha}/{\lambda},\bsgamma^{1/\lambda}\bigr)\right)^{\lambda(1-\delta/2)}
	n^{\lambda(\delta-2)/2-1/2} \\
 \,&\le&\, \left(\frac{c}{\alpha\delta}\right)^{1/2}
 \left(4\,V_d\bigl({\alpha}/{\lambda},\bsgamma^{1/\lambda}\bigr)\right)^{\lambda} \;
 n^{-\lambda-1/2+\delta \lambda/2},
\end{eqnarray*}
which proves the result.
\end{proof}

\section{Lower bound on the error of $M_n$ and $\widetilde{M}_n$}\label{sec:lower}

In the last section we proved that the algorithms $M_n$ and
$\widetilde{M}_n$ can achieve an order of convergence that is arbitrarily
close to the optimal $n^{-\alpha-1/2}$, see~\cite{MU17}. However, we will
show in Theorem~\ref{thm3} that the exact optimal order cannot be achieved
by these algorithms, no matter how we choose the weights $\bsgamma$. This
shows that the main order in our results in Theorems \ref{thm1} and
\ref{thm2} is essentially best possible. Note that the lower bound also
holds for the case $\alpha=0$. However, it would be interesting to find an
upper bound of the form $C(\log n)^q/n^{\alpha+1/2}$ with $C,q$
independent of $n$ and $d$. It is not clear if such a bound exists.

\begin{theorem}\label{thm3}
Let $\alpha\ge 0$ and $\bsgamma \in (0,1]^\NN$. Then the randomized errors
of the random algorithms $M_n$ and $\widetilde{M}_n$ are bounded from
below by
\[
e^{\rm rms}_{d,\alpha,\bsgamma}\bigl(M_n\bigr)
\,\ge\, e^{\rm ran}_{d,\alpha,\bsgamma}\bigl(M_n\bigr)
\,\ge\, \frac{\gamma_1}{2}\,\frac{\sqrt{\log n}}{n^{\alpha+1/2}}
\]
and
\[
e^{\rm rms}_{d,\alpha,\bsgamma}\bigl(\widetilde{M}_n\bigr)
\,\ge\, \frac{\gamma_1}{2}\,\frac{\sqrt{\log n}}{n^{\alpha+1/2}}.
\]
\end{theorem}

\begin{proof}
The bound $e^{\rm rms}_{d,\alpha,\bsgamma}(M_n) \ge e^{\rm
ran}_{d,\alpha,\bsgamma}(M_n)$ is obvious. To prove the lower bound on
$e^{\rm ran}_{d,\alpha,\bsgamma}(M_n)$ it is enough to construct, for each
$n\in\NN$, $n\ge 2$, a function, say $f_n$, that satisfies the lower
bound. To this end, we define $f_n$ such that
\[
\hat{f_n}(\bsh) \,=\, \begin{cases}
\left(r_{\alpha,\bsgamma}(\bsh)\,\sqrt{\abs{\calP_n}}\right)^{-1} &
	\quad\text{ if } h_1\in\calP_n, h_2=\cdots =h_d=0, \\
0 & \quad\text{ otherwise.}
\end{cases}
\]
Clearly, $\|f_n\|_{d,\alpha,\bsgamma}=1$ and $I_d(f_n)=0$. Moreover, for
$p\in\calP_n$, we have
\[\begin{split}
|Q_{d,p,\bsz} (f_n)-I_d (f_n)| \,&=\,
\left(r_{\alpha,\bsgamma}(p,0,\dots,0)\,\sqrt{\abs{\calP_n}}\right)^{-1}
\,=\, \frac{\gamma_1}{p^\alpha \,\sqrt{\abs{\calP_n}}}\\
\,&\ge\, \frac{\gamma_1\,\sqrt{\log n}}{2\,n^{\alpha+1/2}},
\end{split}\]
where we used $\abs{\calP_n}\le\frac{2n}{\log n}$ from~\cite{RS62}.
This proves
\[
e^{\rm ran}_{d,\alpha,\bsgamma}(M_n)
\,\ge\,
 \frac{1}{|\calP_n|} \sum_{p\in\calP_n} \frac{1}{|\calZ_p|} \sum_{\bsz\in\calZ_p}
 |I_d (f_n) - Q_{d,p,\bsz} (f_n)|
\,\ge\, \frac{\gamma_1\,\sqrt{\log n}}{2\,n^{\alpha+1/2}}.
\]
Using
\[
\EE_{\bsU}\left[\left|Q_{d,p,\bsz} \bigl(f_n(\{\cdot+\bsU\})\bigr) - I_d(f_n)\right|^2\right]
\,=\, \left(r_{\alpha,\bsgamma}(p,0,\dots,0)\,\sqrt{\abs{\calP_n}}\right)^{-2},
\]
we can proceed exactly as above to prove the lower bound on $e^{\rm
rms}_{d,\alpha,\bsgamma}\bigl(\widetilde{M}_n\bigr)$.
\end{proof}

\section{Results on tractability}\label{sec:tractresults}

Let us now briefly comment on tractability results. Suppose that for fixed
$d\in\bbN$ and $\varepsilon\in (0,1)$, we would like to achieve
$e_{d,\alpha,\bsgamma}^{\rm ran} (M_n) \le \varepsilon$. Then from
Theorem~\ref{thm1} we conclude that for $\calH_{d,\alpha,\bsgamma}$ with
$\alpha>1/2$, provided $n$ satisfies \eqref{eq:suff-large}, it is
sufficient to choose $n$ such that
\[
 n \,\ge\, \Big(C_{\lambda,\delta} \big[V_d (\alpha/\lambda,\bsgamma^{1/\lambda})\big]^\lambda
 \Big)^{\frac{1}{\lambda+1/2-\delta}}
\,\varepsilon^{-\frac{1}{\lambda+1/2-\delta}}.
\]
Hence, the \emph{information complexity} $n(\varepsilon,d)$, i.e., the
minimal number of function evaluations that is required by any kind of
random algorithm to achieve a randomized error within the
threshold~$\varepsilon$, satisfies
\[
 n(\varepsilon,d)\le \max\left\{\left\lceil 4\,V_d (\alpha/\lambda,\bsgamma^{1/\lambda})\right\rceil,
  \left\lceil \Big(C_{\lambda,\delta} \big[V_d (\alpha/\lambda,\bsgamma^{1/\lambda})\big]^\lambda
  \Big)^{\frac{1}{\lambda+1/2-\delta}}
  \varepsilon^{-\frac{1}{\lambda+1/2-\delta}}\right\rceil\right\}.
\]
More generally, for $\calH_{d,\alpha,\bsgamma}$ with $\alpha>0$, we
obtain from Theorem \ref{thm2} that
\[
n(\varepsilon,d)\le \left\lceil \left( \left(\frac{c}{\alpha\delta}\right)^{1/2}
\big[ 4\,V_d({\alpha}/{\lambda},\bsgamma^{1/\lambda})\big]^{\lambda}
\right)^{\frac{1}{\lambda+1/2-\delta\lambda/2}}\varepsilon^{-\frac{1}{\lambda+1/2-\delta\lambda/2}}\right\rceil.
\]
In both cases, we have $n(\varepsilon,d) = \calO(\varepsilon^{-\beta})$,
for $\beta$ arbitrarily close to $1/(\lambda + 1/2)$, where the implied
constant is independent of $d$ and $\varepsilon$ if $\sum_{j=1}^\infty
\gamma_j^{1/\lambda} < \infty$. The integration problem is then said to be
\emph{strongly tractable in the randomized setting}, with the
\emph{exponent of strong tractability} being the infimum of those $\beta$.

We summarize these results in the following theorem.

\begin{theorem}\label{thm4}
Let $\alpha\ge0$, let $\bsgamma=(\gamma_j)_{j\ge 1}$ be a non-increasing
sequence of positive weights bounded by 1, and let $d\in\NN$.
Additionally, let $\lambda_0\ge0$ be the supremum of the numbers $\lambda>0$
such that
\[
\sum_{j=1}^\infty \gamma_j^{1/\lambda} \,<\, \infty
\]
with the convention that $\lambda_0=0$ if no such $\lambda$ exists.\\
Then the integration problem in the class $\calH_{d,\alpha,\bsgamma}$ is strongly
tractable in the randomized setting, with the exponent of strong tractability
lying in the interval
\[
 \left[\frac{1}{\alpha + 1/2}, \; \frac{1}{\min(\lambda_0,\alpha) + 1/2}\right].
\]
In particular, if $\lambda_0\ge\alpha$, then the exponent of strong tractability is
$1/(\alpha+1/2)$.
\end{theorem}

\begin{proof}
The upper bound on the exponent of strong tractability for $\alpha>0$ and
$\lambda_0>0$ follows from Theorem~\ref{thm2} and the discussion above,
noting that we require $\lambda< \min(\lambda_0,\alpha)$. For $\alpha=0$
it is well-known that the exponent of strong tractability in the
randomized setting is $2$, which can be achieved, e.g., by the simple
Monte Carlo method. For $\lambda_0=0$ we use the same method to obtain
that the exponent of strong
tractability is at most 2.\\
For the lower bound note that $n(\varepsilon,d)\ge n(\varepsilon,1) \ge
c\cdot\varepsilon^{-1/(\alpha+1/2)}$ for some $c>0$, see,
e.g.,~\cite{No88}.
\end{proof}

\begin{remark}
In the same way as it was done in~\cite{SW01}, we could also comment on other forms of
tractability, like polynomial tractability, under modified summability assumptions on the
weights. We omit the details.
\end{remark}

For further information on tractability in various settings and an
overview of various kinds of error criteria in the context of
tractability, see the trilogy \cite{NW08}--\cite{NW12}.

\section{Conclusion}\label{sec:concl}

We showed in this paper that the randomized lattice algorithm $M_n$
achieves a randomized integration error with a convergence order
arbitrarily close to $\calO(n^{-\alpha-1/2})$ in
$\calH_{d,\alpha,\bsgamma}$ with $\alpha>1/2$, where the implied
constant is independent of $d$ as long as $\sum_{j=1}^\infty
\gamma_j^{1/\alpha}<\infty$. By additionally making use of a random
shift, this result can be extended to all $\alpha>0$.

If the weights only satisfy a weaker summability condition
$\sum_{j=1}^\infty \gamma_j^{1/\lambda}<\infty$ for some $\lambda <
\alpha$, then we can still obtain error bounds that are independent of $d$
but the convergence rate is reduced correspondingly to be close to
$\calO(n^{-\lambda-1/2})$.

The question whether component-by-component constructions can be used in
the context of this setting remains open for future research.

\subsection*{Acknowledgement}

P.~Kritzer is supported by the Austrian Science Fund (FWF) Project
F5506-N26, which is part of the Special Research Program ``Quasi-Monte
Carlo Methods: Theory and Applications''. F.Y.~Kuo acknowledges the
financial support from the Australian Research Council (FT130100655 and
DP150101770). D. Nuyens acknowledges the financial support from the KU
Leuven research fund (OT:3E130287 and C3:3E150478).

P.~Kritzer and M.~Ullrich furthermore gratefully acknowledge the partial
support of the Erwin Schr\"odinger International Institute for Mathematics
and Physics (ESI) in Vienna under the thematic programme ``Tractability of
High Dimensional Problems and Discrepancy''.

P.~Kritzer, F.Y.~Kuo, and D.~Nuyens acknowledge the supports from the
Taiwanese National Center for Theoretical Sciences (NCTS) -- Mathematics
Division, and the National Science Foundation Grant DMS-1638521 to the
Statistical and Applied Mathematical Sciences Institute (SAMSI) under its
2017 year-long program on ``Quasi-Monte Carlo and High-Dimensional
Sampling Methods for Applied Mathematics''.

\begin{small}
\noindent\textbf{Authors' addresses:}\\

 \noindent Peter Kritzer\\
 Johann Radon Institute for Computational and Applied Mathematics (RICAM)\\
 Austrian Academy of Sciences\\
 Altenbergerstr. 69, 4040 Linz, Austria.\\
 \texttt{peter.kritzer@oeaw.ac.at}

 \medskip

 \noindent Frances Y. Kuo\\
 School of Mathematics and Statistics\\
 University of New South Wales\\
 Sydney, NSW, 2052, Australia.\\
 \texttt{f.kuo@unsw.edu.au}

 \medskip

 \noindent Dirk Nuyens\\
 Department of Computer Science\\
 KU Leuven\\
 Celestijnenlaan 200A, 3001 Leuven, Belgium.\\
 \texttt{dirk.nuyens@cs.kuleuven.be}

 \medskip

 \noindent Mario Ullrich\\
 Department of Analysis\\
 JKU Linz\\
 Altenbergerstr. 69, 4040 Linz, Austria.\\
 \texttt{mario.ullrich@jku.at}

 \end{small}

\end{document}